\documentclass[english,smallextended,numbook,envcountsame,final]{svjour3}
\usepackage[utf8]{inputenc}
\usepackage{babel}
\usepackage{array}
\usepackage{units}
\usepackage{url}
\usepackage{amsmath}
\usepackage{amssymb}
\usepackage[unicode=true,
 bookmarks=true,bookmarksnumbered=false,bookmarksopen=false,
 breaklinks=true,pdfborder={0 0 0},backref=false,colorlinks=false]
 {hyperref}
\hypersetup{pdftitle={Some graph theoretical characterizations of positive definite symmetric quasi-Cartan matrices},
 pdfauthor={M. Abarca, D. Rivera}}

\makeatletter

\providecommand{\tabularnewline}{\\}

\newenvironment{svmultproof}{\begin{proof}}{\qed\end{proof}}

\usepackage{mathptmx} %use PostScript Times font
\usepackage{tikz} %package for drawing graphs
\smartqed  % flush right qed marks, e.g. at end of proof
\titlerunning{Graphical characterizations of sym. quasi-Cartan matrices} %Short title

\makeatother

\begin{document}
\global\long\def\Z{\mathbb{Z}}
\global\long\def\N{\mathbb{N}}
\global\long\def\T{\mathrm{T}}
\global\long\def\A{\mathbb{A}}
\global\long\def\B{\mathbb{B}}
\global\long\def\C{\mathbb{C}}
\global\long\def\D{\mathbb{D}}
\global\long\def\E{\mathbb{E}}
\global\long\def\Dyn{\mathrm{Dyn}}
\global\long\def\F{\mathbf{F}}
\global\long\def\blocks{\mathcal{B}}
\global\long\def\solid{\mathtt{solid}}
\global\long\def\dotted{\mathtt{dotted}}
\global\long\def\ls{\lambda}
\global\long\def\BT{\mathrm{BT}}
\global\long\def\vertexset{\mathcal{V}}
\global\long\def\R{\mathbb{R}}
\global\long\def\col{\mathrm{col}}
\global\long\def\row{\mathrm{row}}
\global\long\def\Gr{\mathbf{G}}
\global\long\def\Ad{\mathbf{A}}
\global\long\def\dist{d}
\global\long\def\ident{\mathbf{1}}
\global\long\def\K{\mathbf{K}}

\title{Some graph theoretical characterizations of positive definite symmetric
quasi-Cartan matrices%
\thanks{The authors gratefully acknowledge the support of CONACyT (Grant 156667).%
}}

\author{M. Abarca \and D. Rivera}

\institute{M. Abarca \and D. Rivera \at Av. Universidad 1001, Col. Chamilpa
62209, Cuernavaca, Mor. MEX.\\
Tel.: +52-777-329-7020 ext. 3678\\
\email{asma@uaem.mx, darivera@uaem.mx} }

\date{}
\maketitle
\begin{abstract}
We present some graphical characterizations of positive definite symmetric
quasi-Cartan matrices of Dynkin type $\A_{n}$ and $\D_{n}$. Our
proofs are constructive, purely graph theoretical, and almost self-contained
in the sense that they rely on the classical inflations method only.\keywords{Graph connectivity\and Dynkin type\and quasi-Cartan matrix}
\end{abstract}

\section{Introduction\label{sec:Introduction}}

In this paper we consider the symmetric quasi-Cartan matrices. These
are symmetric matrices $A\in\mathcal{M}_{n}\left(\Z\right)$ with
all diagonal entries $A_{i\, i}=2$. If a quasi-Cartan matrix is also
positive definite and all non-diagonal entries $A_{i\, j}<0$ then
it is a (symmetric) Cartan matrix. Two symmetric quasi-Cartan matrices
$A$ and $A^{\prime}$ are said to be $\Z$-equivalent if $A^{\prime}=M^{\T}\, A\, M$
for some matrix $M\in\mathcal{M}_{n}\left(\Z\right)$ with $M^{-1}\in\mathcal{M}_{n}\left(\Z\right)$.

We adapt the graph terminology of \cite{Bondy2008}. Any symmetric
quasi-Cartan matrix $A$ can be depicted as being an adjacency-like
matrix of a graph $G=\left(V\left(G\right),E\left(G\right)\right)$
with vertices $V\left(G\right)=\left\{ 1,\ldots,n\right\} $. For
each pair of $\left\{ i,j\right\} \in\binom{V\left(G\right)}{2}$
with $A_{i\, j}\ne0$ add edges $e_{1},\ldots,e_{\left|A_{i\, j}\right|}$
to $E\left(G\right)$ with endpoints $\psi\left(e_{k}\right)=\left\{ i,j\right\} $,
and a \emph{line style} $\lambda:E\to\left\{ \solid,\dotted\right\} $
given by $\lambda\left(e_{k}\right)=\dotted$ if $A_{i\, j}>0$, and
$\solid$ otherwise. Such graph is called the \emph{bigraph} of $A$
and we will denote it as $\Gr_{A}$. Likewise, we write $\Ad_{G}$
for the quasi-Cartan matrix associated to the bigraph $G$. When restricted
to symmetric Cartan matrices, each connected component of a bigraph
is one of the Dynkin diagrams depicted in Figure \ref{fig:DynkinDiagrams}
(for a reference, see \cite[chap. II sec. 5]{Knapp2002}).

\begin{figure}
\begin{centering}
\begin{tabular}{ll}
Family & Graph\tabularnewline
$\A_{n}$, $n\ge1$ & \begin{tikzpicture}
[baseline=(v1).base, scale = 0.75, line width=0.5pt, every node/.style={circle, draw, inner sep=1pt}]
    \node (v1) at (0, 0) {};
    \node (v2) at (1, 0) {};
	\node (v3) at (4, 0) {};
	\node (v4) at (5, 0) {};
	\node[draw = none] (v5) at (2.5, 0) {$\ldots$};
	\draw (v1) -- (v2) -- (2, 0);
	\draw (3, 0) -- (v3) -- (v4);
\end{tikzpicture}\tabularnewline
$\D_{n}$, $n\ge4$ & \begin{tikzpicture}
[baseline=(v1).base, scale = 0.75, line width=0.5pt, every node/.style={circle, draw, inner sep=1pt}]
    \node (v1) at (0, 0) {};
    \node (v2) at (1, 0) {};
	\node (v3) at (2, 0) {};
	\node (v4) at (5, 0) {};
	\node (v5) at (1, 1) {};
	\node[draw = none] (dots) at (3.5, 0) {$\ldots$};
	\draw (v1) -- (v2) -- (v3) -- (3, 0);
	\draw (v5) -- (v2);
	\draw (4, 0) -- (v4);
\end{tikzpicture}\tabularnewline
$\E_{n}$, $6\le n\le8$ & \begin{tikzpicture}
[baseline=(v1).base, scale = 0.75, line width=0.5pt, every node/.style={circle, draw, inner sep=1pt}]
    \node (v1) at (0, 0) {};
    \node (v2) at (1, 0) {};
	\node (v3) at (2, 0) {};
	\node (v4) at (3, 0) {};
	\node (v5) at (6, 0) {};
    \node (v6) at (2, 1) {};
	\node[draw = none] (dots) at (4.5, 0) {$\ldots$};
	\draw (v1) -- (v2) -- (v3) -- (v4) -- (4, 0);
	\draw (v6) -- (v3);
	\draw (5, 0) -- (v5);
\end{tikzpicture}\tabularnewline
\end{tabular}
\par\end{centering}

\caption{\label{fig:DynkinDiagrams}Dynkin diagrams. The parameter $n$ denotes
the number of vertices in the graph.}
\end{figure}
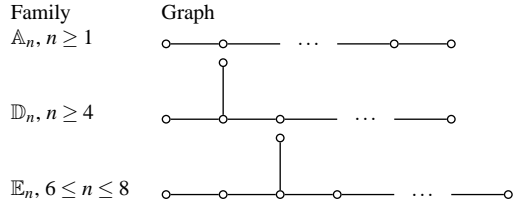

\section{The inflations method on simple bigraphs}

For any $\sigma\in\R$ and $\left(s,r\right)\in\left\{ 1,\ldots,n\right\} ^{2}$
let $M_{s\, r}^{\sigma}=\ident+\sigma\,\vec{e}_{s}\,\vec{e}_{r}^{\T}$,
where $\left(\vec{e}_{1},\ldots,\vec{e}_{n}\right)$ is the canonical
basis of $\R^{n}$ and $\ident$ is the identity matrix of size $n$.
 In \cite[sec. 6.2]{Gabriel1997} it is shown that if $A$ is any
positive definite symmetric quasi-Cartan matrix then:
\begin{enumerate}
\item $A_{i\, j}\in\left\{ -1,0,1\right\} $ for all $i\ne j$. Therefore
$\Gr_{A}$ must be a simple bigraph and any edge $e$ may be uniquely
identified with its endpoints (i.e. $e=\psi\left(e\right)$).
\item $A$ is $\Z$-equivalent to a Cartan matrix. This property is proven
using the \emph{inflations method}: start with $A^{\left(0\right)}=A$,
then set $A^{\left(k+1\right)}=\left(M_{s\, r}^{-1}\right)^{\T}\, A^{\left(k\right)}\, M_{s\, r}^{-1}$
whenever an entry $A_{s\, r}^{\left(k\right)}=1$ exists for each
$k=1,2,3,\ldots$ until no such entry exists. The last $A^{\left(\ell\right)}$
computed is guarantied to be a Cartan matrix.%
\footnote{This results are expressed in \cite{Gabriel1997} in terms of the
associated quadratic form $q\left(\vec{x}\right)=\frac{1}{2}\,\vec{x}^{\T}\, A\,\vec{x}$.%
} The matrix $\Dyn\left(A\right)=A^{\left(\ell\right)}$ is unique
up to isomorphism and may be called the \emph{Dynkin type} of $A$
(two matrices $A$ and $A^{\prime}$ being isomorphic if $A^{\prime}=P^{\T}\, A\, P$
for some permutation matrix $P$). 
\end{enumerate}
We transfer this terminology to bigraphs by saying that a bigraph
$G=\Gr_{A}$ has Dynkin type $\Delta=\Gr_{B}$ if $A$ is $\Z$-equivalent
to the Cartan matrix $B$. Moreover, let $T_{s\, r}\,\Gr_{A}=\Gr_{A^{\prime}}$
where $A^{\prime}=\left(M_{s\, r}^{-A_{s\, r}}\right)^{\T}\, A\, M_{s\, r}^{-A_{s\, r}}$.
In the following we give characterizations of graphs with Dynkin types
$\A_{n}$ and $\D_{n}$ which proof is graphical, constructive and
self-contained. In doing so, we will make use of the following basic
lemmas:
\begin{lemma}
\label{lem:flations_method}The bigraph $G$ has Dynkin type $\Delta$
if and only if $G=T_{s_{\ell}\, r_{\ell}}\cdots T_{s_{2}\, r_{2}}\, T_{s_{1}\, r_{1}}\,\Delta$
for some finite sequence $\left(s_{k},r_{k}\right)_{k=1}^{\ell}$.\end{lemma}
\begin{svmultproof}
Immediate by the inflations method.\end{svmultproof}
\begin{lemma}
\label{lem:deflation_computation}Let $G$ be a simple bigraph. Then
$T_{s\, r}\, G$ can be computed as follows:
\begin{enumerate}
\item If $G$ does not contains the edge $\left\{ s,r\right\} $ do nothing
(ignore steps 2 through 4).
\item Change the line style of the edge $\left\{ s,r\right\} $ to the opposite
style ($\dotted$ and $\solid$ being opposites).
\item If edge $\left\{ s,r\right\} $ is $\solid$ (resp. $\dotted$) in
$G$, then for each neighbor $i$ of $s$ add an edge between vertices
$i$ and $r$ with the same (resp. opposite) line style as the edge
$\left\{ i,s\right\} $, and then cancel out pairs of parallel edges
between $i$ and $s$ whenever they have opposite line style.
\item All other edges of $G$ remain the same.
\end{enumerate}
\end{lemma}
\begin{svmultproof}
A direct computation of $A^{\prime}=\left(M_{s\, r}^{\sigma}\right)^{\T}\, A\, M_{s\, r}^{\sigma}$
where $A$ is any square matrix yields:

\[
A_{i\, j}^{\prime}=\begin{cases}
A_{i\, j} & \text{if }i\ne r\text{ and }j\ne r\\
A_{r\, j}+\sigma\, A_{s\, j} & \text{if }i=r\text{ and }j\ne r\\
A_{i\, r}+\sigma\, A_{i\, s} & \text{if }i\ne r\text{ and }j=r\\
A_{r\, r}+\sigma\left(A_{s\, r}+A_{r\, s}\right)+\sigma^{2}\, A_{s\, s} & \text{if }i=r\text{ and }j=r
\end{cases}
\]
If we let $\sigma=-A_{s\, r}$, $A_{s\, r}=\pm1$, $A_{i\, j}=A_{j\, i}$
and $A_{i\, i}=2$ for all $i,j$, then $A_{r\, r}^{\prime}=2$ and
$A_{i\, j}^{\prime}=A_{j\, i}^{\prime}$, so that $A^{\prime}$ is
a symmetric quasi-Cartan matrix with properties: (1) if $A_{s\, r}=0$
then $A^{\prime}=A$; (2) $A_{s\, r}^{\prime}=-A_{s\, r}$; (3) $A_{i\, r}^{\prime}=A_{i\, r}\mp A_{i\, s}$
for all $i<r$; and (4) $A_{i\, j}^{\prime}=A_{i\, j}$ for all $i\le j$
with $j\ne r$.
\end{svmultproof}

\section{The Dynkin type $\A_{n}$ bigraphs}

The main theorem of \cite{Barot1999} is a characterization of the
Dynkin type $\A_{n}$ bigraphs in terms of the construction of such
bigraphs. We present a similar characterization with a constructive
proof in terms of its \emph{block tree} (i.e. its graph decomposition
up to biconnected components) as described by \cite[sec. 5.2]{Bondy2008}:
A block tree of any graph $G$, denoted here as $\BT\left(G\right)$,
is the bipartite graph with vertices $\blocks\cup S$ where $\blocks$
and $S$ are the set of blocks and the set of separating vertices
of $G$ respectively, and tree edges given by $\left\{ \left\{ B,s\right\} \middle|B\in\blocks,s\in V\left(B\right)\right\} $.
Thus, if $G$ is a simple connected graph then the degree of any separation
vertex $s$ of $G$ in the block tree equals the number of connected
components obtained after removing $s$ (along with its incident edges).

Let $\F\left[X,Y\right]$ be the bigraph obtained by joining each
pair $\left\{ x,y\right\} $ of $x\in X$ and $y\in Y$ by a $\dotted$
edge, and all other pairs $\left\{ x,x^{\prime}\right\} \in\binom{X}{2}$
and $\left\{ y,y^{\prime}\right\} \in\binom{Y}{2}$ by a $\solid$
edge. 
\begin{lemma}
\label{lem:An_deflation}Let $G=\F\left[X,Y\right]\cup\F\left[X^{\prime},Y^{\prime}\right]$
where $X$, $Y$, $X^{\prime}$ and $Y^{\prime}$ are pairwise disjoint
except for $X\cap X^{\prime}=\left\{ s\right\} $, then:
\begin{enumerate}
\item $T_{s\, r}\, G=\F\left[X,Y\setminus\left\{ r\right\} \right]\cup\F\left[X^{\prime}\cup\left\{ r\right\} ,Y^{\prime}\right]$
for any $r\in Y$, and
\item $T_{s\, r}\, G=\F\left[X\setminus\left\{ r\right\} ,Y\right]\cup\F\left[X^{\prime},Y^{\prime}\cup\left\{ r\right\} \right]$
for any $r\in X\setminus\left\{ s\right\} $.
\end{enumerate}
\end{lemma}
\begin{svmultproof}
Use Lemma \ref{lem:deflation_computation} to compute $G^{\prime}=T_{s\, r}\, G$.
First assume $r\in Y$. Notice that for each $i\in X\setminus\left\{ s\right\} $
there is a $\dotted$ edge $\left\{ s,i\right\} $ that cancels out
the $\solid$ edge $\left\{ i,r\right\} $ and for each $i\in Y$
there is a $\solid$ edge $\left\{ s,i\right\} $ that cancels out
the $\dotted$ edge $\left\{ i,r\right\} $, so that the vertex $r$
is not adjacent in $G^{\prime}$ to any other vertex of $X$ nor $Y$,
except for $s$, which is connected to $r$ by a $\dotted$ edge.
Also, since $r$ is not adjacent to any vertex of $X^{\prime}$ nor
$Y^{\prime}$ in $G$, then for each $i\in X^{\prime}$ (resp. $i\in Y^{\prime}$)
we add a $\dotted$ (resp. $\solid$) edge $\left\{ i,r\right\} $.
By definition we get $G^{\prime}=\F\left[X,Y\setminus\left\{ r\right\} \right]\cup\F\left[X^{\prime}\cup\left\{ r\right\} ,Y^{\prime}\right]$.
The case where $r\in X\setminus\left\{ s\right\} $ is analog.
\end{svmultproof}
\begin{figure}
\begin{centering}
\begin{tikzpicture}
[baseline=(s).base, scale = 1, line width=0.5pt, every node/.style={circle, inner sep=1pt}]
    \node (s) at (1, 1) {$s$};
    \node (r) at (3, 1) {$r$};
	\node (x) at (2, 2) {$x$};
	\node (y) at (2, 0) {$y$};
	\node (xp) at (0, 2) {$x^\prime$};
    \node (yp) at (0, 0) {$y^\prime$};
	\draw (xp) -- (yp) -- (s) -- (r) -- (x) -- (y) -- (s);
	\draw[dotted] (xp) -- (s) -- (x);
	\draw[dotted] (y) -- (r);
\end{tikzpicture}$\qquad\overset{T_{s\, r}}{\longmapsto}\qquad$\begin{tikzpicture}
[baseline=(s).base, scale = 1, line width=0.5pt, every node/.style={circle, inner sep=1pt}]
    \node (s) at (2, 1) {$s$};
    \node (r) at (0, 1) {$r$};
	\node (x) at (3, 2) {$x$};
	\node (y) at (3, 0) {$y$};
	\node (xp) at (1, 2) {$x^\prime$};
    \node (yp) at (1, 0) {$y^\prime$};
	\draw (r) -- (yp) -- (s) -- (y) -- (x);
	\draw (yp) -- (xp);
	\draw[dotted] (s) -- (r) -- (xp) -- (s) -- (x);
\end{tikzpicture}
\par\end{centering}

\caption{\label{fig:An_Tsr_example}A minimal nontrivial example of Lemma \ref{lem:An_deflation},
where $X=\left\{ x,s\right\} $, $Y=\left\{ y,r\right\} $, $X^{\prime}=\left\{ x^{\prime},s\right\} $
and $Y^{\prime}=\left\{ y^{\prime}\right\} $.}
\end{figure}

For simplicity, we say that two bigraphs $G$ and $G^{\prime}$ are
joined by a vertex $s$ if $V\left(G\right)\cap V\left(G^{\prime}\right)=\left\{ s\right\} $.
If we write $G\cong\F_{n}$ to mean that $G=\F\left[X,Y\right]$ for
some $n=\left|X\right|+\left|Y\right|$ then the previous lemma can
be succinctly rewritten as follows: 
\begin{corollary}
\label{cor:joined_F_deflation}Let $B_{1}\cong\F_{n}$ and $B_{2}\cong\F_{m}$
be joined by a vertex $s$, and let $r\in V\left(B_{1}\right)$. Then
$T_{s\, r}\,\left(B_{1}\cup B_{2}\right)=B_{1}^{\prime}\cup B_{2}^{\prime}$
where $B_{1}^{\prime}\cong\F_{n-1}$ and $B_{2}^{\prime}\cong\F_{m+1}$
are joined by $s$, and $r\in V\left(B_{2}^{\prime}\right)\setminus V\left(B_{1}^{\prime}\right)$.
\end{corollary}
We say that the block tree $\BT\left(G\right)$ is an \emph{$\A$-block
tree} if each block $B_{i}\cong\F_{n_{i}}$ for some $n_{i}$ and
each separation vertex $s$ of $G$ has degree $2$ in $\BT\left(G\right)$
(so that $s$ joins exactly two blocks of $G$).
\begin{lemma}
\label{lem:TpreservesABT}If $\BT\left(G\right)$ is an $\A$-block
tree then $\BT\left(T_{s\, r}\, G\right)$ is an $\A$-block tree.\end{lemma}
\begin{proof}
Let $G^{\prime}=T_{s\, r}\, G$ and suppose that $\BT\left(G\right)$
is an $\A$-block tree. If $G$ does not contain the edge $\left\{ s,r\right\} $
then $G^{\prime}=G$ and the lemma is trivially true; otherwise let
$B$ be the block of $G$ containing edge $\left\{ s,r\right\} $
and consider the bigraph $\bar{B}$ induced by $s$ and all neighbors
of $s$ that do not belong to $B$. We have two cases: 
\begin{enumerate}
\item either $s$ is not a separation vertex, in which case $B\cong\F_{m}$
and $\bar{B}\cong\F_{m^{\prime}}$ for some $m\ge2$ and $m^{\prime}=1$,
or 
\item $s$ is a separation vertex, and $B\cong\F_{m}$, $\bar{B}\cong\F_{m^{\prime}}$
for some $m\ge2$ and $m^{\prime}>1$.
\end{enumerate}
Since $B$ and $\bar{B}$ are joined by $s$, then by Corollary \ref{cor:joined_F_deflation}
we have $T_{s\, r}\left(B\cup\bar{B}\right)=B^{\prime}\cup\bar{B}^{\prime}$
where $B^{\prime}\cong\F_{m-1}$, $\bar{B}^{\prime}\cong\F_{m^{\prime}+1}$,
and $r\in V\left(\bar{B}^{\prime}\right)$. Using property 4 of Lemma
\ref{lem:deflation_computation} we can compute $\BT\left(G^{\prime}\right)$
from $\BT\left(G\right)$ as follows:
\begin{enumerate}
\item Let $\mathcal{T}=\left(\mathcal{V}\cup\left\{ s,\bar{B}\right\} ,\mathcal{E}\cup\left\{ \left\{ B,s\right\} ,\left\{ s,\bar{B}\right\} \right\} \right)$
where $\BT\left(G\right)=\left(\mathcal{V},\mathcal{E}\right)$. Note
that $\deg_{\mathcal{T}}\left(s\right)=2$, and that $\mathcal{T}=\BT\left(G\right)$
if and only if $m^{\prime}>1$.
\item If $m=2$ then $B^{\prime}$ is not a block of $G^{\prime}$ and $s$
is not a separation vertex of $G^{\prime}$. In this case let $\mathcal{T}^{\prime}$
be the result of shrinking $B$, $s$ and $\bar{B}$ into a single
vertex $\bar{B}^{\prime}$ in $\mathcal{T}$.
\item If $m>2$ then both $B^{\prime}$ and $\bar{B}^{\prime}$ are blocks
of $G^{\prime}$ joined by the separation vertex $s$. In this case
let $\mathcal{T}^{\prime}$ be the result of $\mathcal{T}$ after
removing its edge $\left\{ r,B\right\} $, replacing $B$ and $\bar{B}$
with $B^{\prime}$ and $\bar{B}^{\prime}$ respectively, and then
adding the edge $\left\{ r,\bar{B}^{\prime}\right\} $.\qed
\end{enumerate}
\end{proof}
\begin{theorem}
\label{teo:An_iff_Ablocktrees}$\Dyn\left(G\right)=\A_{\left|V\left(G\right)\right|}$
if and only if $\BT\left(G\right)$ is an $\A$-block tree.\end{theorem}
\begin{svmultproof}
$\Rightarrow$) By induction on $k$ of Lemma \ref{lem:flations_method}.
For the base case ($G^{\left(0\right)}=\A_{n}$): \newcommand{\AnV}{\begin{tikzpicture}
[baseline=-0.5ex, scale = 1, line width=0.5pt, every node/.style={circle, inner sep=0pt}]
    \node (v1) at (0, 0) {$v_1$};
    \node (v2) at (1, 0) {$v_2$};
	\node (v3) at (4, 0) {$v_n$};
	\node[draw = none] (v5) at (2.5, 0) {$\ldots$};
	\draw (v1) -- (v2) -- (2, 0);
	\draw (3, 0) -- (v3) ;
\end{tikzpicture}} \newcommand{\AnBT}{\begin{tikzpicture}
[baseline=-0.5ex, scale = 1, line width=0.5pt, every node/.style={circle, inner sep=0pt}]
    \node (B1) at (0, 0) {$B_1$};
    \node (v2) at (1, 0) {$v_2$};
	\node (B2) at (2, 0) {$B_2$};
	\node[draw = none] (dots) at (3.5, 0) {$\ldots$};
	\node (Bn2) at (5, 0) {$B_{n-2}$};
	\node (vn1) at (6, 0) {$v_{n-1}$};
	\node (Bn1) at (7, 0) {$B_{n-1}$};
	\draw (B1) -- (v2) -- (B2) -- (3,0);
	\draw (4,0) -- (Bn2) -- (vn1) -- (Bn1);
\end{tikzpicture}}
\begin{eqnarray*}
\A_{n} & = & \AnV\\
\BT\left(\A_{n}\right) & = & \AnBT
\end{eqnarray*}
where each block $B_{i}=\F\left[\left\{ v_{i}\right\} ,\left\{ v_{i+1}\right\} \right]\cong\F_{2}$
for $i\in\left\{ 1,\ldots,n-1\right\} $, and each separation vertex
$v_{2},\ldots,v_{n-1}$ clearly has degree $2$ in $\BT\left(\A_{n}\right)$.
Therefore $\BT\left(\A_{n}\right)$ is an $\A$-block tree. For the
induction step, suppose that $\BT\left(G^{\left(k-1\right)}\right)$
is an $\A$-block tree and compute $G^{\prime}=T_{s\, r}\, G^{\left(k-1\right)}$.
Then by Lemma \ref{lem:TpreservesABT} $\BT\left(G^{\left(k\right)}\right)$
is also an $\A$-block tree.

$\Leftarrow$) Let $\mathcal{T}=\BT\left(G\right)$ and $n=\left|V\left(G\right)\right|$.
We show how to transform $G$ into $H\cong\F_{n}$, and then $H$
into $\A_{n}$ using $T$ transformations only. Suppose $G\ncong\F_{n}$
then $\mathcal{T}$ contains a path $B_{1}\, s\, B_{2}$ where $B_{1}\cong\F_{n_{1}}$,
$B_{2}\cong\F_{n_{2}}$, and $B_{1}$ is a leaf block of $G$ (i.e.
$\deg_{\mathcal{T}}\left(B_{1}\right)=1$). For simplicity, let $V\left(B_{1}\right)=\left\{ v_{1},\ldots,v_{n_{1}}\right\} $
and $B_{3}=T_{s\, v_{1}}\, T_{s\, v_{2}}\,\ldots\, T_{s\, v_{n}}\left(B_{1}\cup B_{2}\right)$.
Using Corollary \ref{cor:joined_F_deflation} repeatedly we see that
$B_{3}\cong\F_{n_{1}+n_{2}}$. Thus, if $G^{\prime}=T_{s\, v_{1}}\,\ldots\, T_{s\, v_{n}}\, G$,
then $\BT\left(G^{\prime}\right)$ can be obtained from $\BT\left(G\right)$
by shrinking $B_{1}$, $s$ and $B_{2}$ into a new vertex $B_{3}$.
Repeating this shrinking operation we arrive to a bigraph $H\cong\F_{n}$.
Finally, let $H=\F\left[\left\{ x_{1},\ldots,x_{m}\right\} ,\left\{ y_{1},\ldots,y_{m^{\prime}}\right\} \right]$
and 
\[
H^{\prime}=T_{x_{m}\, x_{m-1}}\,\ldots\, T_{x_{3}\, x_{2}}\, T_{x_{2}\, x_{1}}\, T_{y_{1}\, y_{2}}\, T_{y_{2}\, y_{3}}\,\ldots\, T_{y_{m^{\prime}-1}\, y_{m^{\prime}}}\, H
\]
Then $H^{\prime}=\A_{n}$ with $v_{i}=x_{i}$ for $i\in\left\{ 1,\ldots,m\right\} $
and $v_{j+m}=y_{j}$ for $j\in\left\{ 1,\ldots,m^{\prime}\right\} $.
\end{svmultproof}

\section{The Dynkin type $\D_{n}$ bigraphs}

Next, we give a characterization of the bigraphs of Dynkin type $\D_{n}$.
In \cite{Barot2001} two types of constructions are suggested. We
present yet another construction which proof follows the same pattern
as the $\A_{n}$ case, and then use it to prove another construction
that resembles the one founded in \cite{Barot2001}. For our construction
we need:
\begin{itemize}
\item A cycle bigraph $H=\left(\left\{ x_{1},x_{2},\ldots,x_{h}\right\} ,\left\{ e_{1},e_{2},\ldots,e_{h}\right\} \right)$
with $h\ge2$, defined by $\psi\left(e_{i}\right)=\left\{ x_{i},x_{i+1}\right\} $
for $i\in\left\{ 1,\ldots,h\right\} $ ($x_{h+1}=x_{1}$) and such
that $H$ contains an odd number of $\dotted$ edges.
\item Bigraphs $F_{1},F_{2},\ldots,F_{h}$ where each $F_{i}=\left(V_{i},E_{i}\right)$
has Dynkin type $\A_{n_{i}}$, contains the edge $e_{i}$ with the
same line style as in $H$ (i.e. $\ls_{F_{i}}\left(e_{i}\right)=\ls_{H}\left(e_{i}\right)$),
and neither $x_{i}$ nor $x_{i+1}$ are separation vertices of $F_{i}$.
We also require that all $F_{i}$'s are vertex pairwise disjoint except
for the endpoints of $e_{i}$; more precisely $V_{i}\cap V_{i+1}=\left\{ x_{i+1}\right\} $
for $i\in\left\{ 1,\ldots,h-1\right\} $, $V_{h}\cap V_{1}=\left\{ x_{1}\right\} $,
and $V_{i}\cap V_{j}=\emptyset$ for all other pairs of $i,j$.
\end{itemize}
Taking the union of all $F_{i}$ we get a new bigraph $\hat{H}$ which
by construction contains a copy of $H$. Call the simplified version
of $\hat{H}$ (where each pair of parallel edges with opposite line
styles are erased) the \emph{$\D$-cycle gluing} of $H$ and $F_{1},F_{2},\ldots,F_{h}$.
\begin{lemma}
\label{lem:T_preserves_Dcycle}If $G$ is a $\D$-cycle gluing, then
$T_{s\, r}\, G$ is a $\D$-cycle gluing.\end{lemma}
\begin{proof}
Let $G$ be the $\D$-cycle gluing of $H$ and $F_{1},F_{2},\ldots,F_{h}$,
and let $G^{\prime}=T_{s\, r}\, G$. If $\left\{ s,r\right\} \notin E\left(G\right)$
then the lemma is trivially true. Otherwise for $i\in\left\{ 1,\ldots,h\right\} $
let $B_{i}\cong\F_{m_{i}}$ be the block of $F_{i}$ which contains
the edge $\left\{ x_{i},x_{i+1}\right\} $. Since no $x_{i}$ is a
separation vertex of $F_{i-1}$ nor $F_{i}$, then $F_{i-1}\cup F_{i}$
has Dynkin type $\A_{n_{i-1}+n_{i}}$. Taking this into account, there
are four possibilities for $s$ and $r$:
\begin{enumerate}
\item Neither $s$ or $r$ are members of $V\left(H\right)$. This case
does not affect the cycle bigraph $H$, since $\left\{ s,r\right\} \in E\left(F_{i}\right)\setminus E\left(H\right)$
for some $i$, and therefore $G^{\prime}$ is the $\D$-cycle gluing
of $H$ and $F_{1},\ldots,F_{i-1},T_{s\, r}\, F_{i},F_{i+1},\ldots,F_{h}$.
\item Only $s$ is a member of $V\left(H\right)$. Assume $s=x_{i}$. We
prove that $T_{s\, r}$ moves $r$ from $F_{i-1}$ to $F_{i}$ or
from $F_{i}$ to $F_{i-1}$ without affecting $H$. Either $\left\{ s,r\right\} \in E\left(B_{i}\right)$
or $\left\{ s,r\right\} \in E\left(B_{i-1}\right)$. Without any loss
of generality let $\left\{ s,r\right\} \in B_{i-1}$, so that $m_{i-1}\ge3$.
Then by Corollary \ref{cor:joined_F_deflation} $T_{s\, r}\,\left(B_{i-1}\cup B_{i}\right)$
has the effect of “moving” the vertex $r$ from $B_{i-1}$ to $B_{i}$,
yielding the blocks $B_{i-1}^{\prime}\cong\F_{m_{i-1}-1}$ and $B_{i}^{\prime}\cong\F_{m_{i}+1}$.
This effectively translates to moving $r$ from $F_{i-1}$ to $F_{i}$
yielding some bigraphs $F_{i-1}^{\prime}$ and $F_{i}^{\prime}$.
Note that $\left\{ x_{i-1},x_{i}\right\} \in E\left(B_{i-1}^{\prime}\right)$
and $\left\{ x_{i},x_{i+1}\right\} \in E\left(B_{i-1}^{\prime}\right)$;
thus $G^{\prime}$ is the $\D$-cycle gluing of $H$ and $F_{1},\ldots,F_{i-2},F_{i-1}^{\prime},F_{i}^{\prime},F_{i+1},\ldots,F_{h}$.
\item Only $r$ is a member of $V\left(H\right)$. We prove that $T_{s\, r}$
grows the cycle bigraph by absorbing $s$ into it. Let $r=x_{i}$;
either $\left\{ s,r\right\} \in B_{i}$ or $\left\{ s,r\right\} \in B_{i-1}$.
Assume without loss of generality that $\left\{ s,r\right\} \in B_{i}$,
so that $m_{i}\ge3$. Let $\bar{B}_{i}$ be the subgraph of $F_{i}$
induced by $s$ and all neighbors of $s$ not in $V\left(B_{i}\right)$.
Then $\bar{B}_{i}\cong\F_{m_{i}^{\prime}}$ where $m_{i}^{\prime}\ge1$
with $m_{i}^{\prime}>1$ if and only if $s$ is a separation vertex
of $G$. By Corollary \ref{cor:joined_F_deflation} $T_{s\, r}\,\left(B_{i}\cup\bar{B}_{i}\right)=B_{i}^{\prime}\cup\bar{B}_{i}^{\prime}$
where $B_{i}^{\prime}\cong\F_{m_{i}-1}$, $\bar{B}_{i}^{\prime}\cong\F_{m_{i}^{\prime}+1}$,
$r\in V\left(\bar{B}_{i}^{\prime}\right)\setminus V\left(B_{i}^{\prime}\right)$,
and $s$ joins $B_{i}^{\prime}$ and $\bar{B}_{i}^{\prime}$. Then
$s$ is a separation vertex of $T_{s\, r}\, F_{i}$ which joins two
sub-bigraphs $F_{i}^{\prime}$ (the one containing $B_{i}^{\prime}$)
and $\bar{F}_{i}^{\prime}$ (the one containing $\bar{B}_{i}^{\prime}$).
By property 4 of Lemma \ref{lem:deflation_computation} we see that$\left\{ x_{i-1},r\right\} \in E\left(G^{\prime}\right)$,
so that $G^{\prime}$ contains the path $x_{i-1}\, r\, s\, x_{i+1}$
but not the edge $\left\{ r,x_{i+1}\right\} $. Notice that $\ls_{G}\left(\left\{ s,x_{i+1}\right\} \right)=\solid$
if and only if $\ls_{G^{\prime}}\left(\left\{ r,s\right\} \right)=\ls_{G^{\prime}}\left(\left\{ s,x_{i+1}\right\} \right)$.
Therefore, if we define $H^{\prime}$ as the cycle graph given by
the walk $x_{1}\, x_{2}\,\ldots\, x_{i}\, s\, x_{i+1}\,\ldots\, x_{h}\, x_{1}$
in $G^{\prime}$, then $H^{\prime}$ has an odd number of $\dotted$
edges if and only if $H$ does, and so $G^{\prime}$ is the $\D$-cycle
gluing of $H$ and $F_{1},\ldots,F_{i-1},F_{i}^{\prime},\bar{F}_{i}^{\prime},F_{i+1},\ldots,F_{h}$.
\item Both $s$ and $r$ are members of $V\left(H\right)$. Let $r=x_{i}$
and $s=x_{i+1}$. This case has the effect of shrinking $H$ by joining
$F_{i}$ and $F_{i+1}$ and removing $s$ from $H$. To see this,
notice first that if $V\left(H\right)=\left\{ s,r\right\} $ then
by property 1 of Lemma \ref{lem:deflation_computation} $T_{s\, r}\, G=G$,
therefore we may assume that $\left|V\left(H\right)\right|\ge3$.
Since two consecutive applications of $T_{s\, r}$ annihilate (i.e
$T_{s\, r}\, T_{s\, r}$ is the identity transformation) and since
$G^{\prime}$ of the previous case contains $\left\{ s,r\right\} \in V\left(H^{\prime}\right)$
we infer that $T_{s\, r}$ has the opposite effect as previously stated.
\qed
\end{enumerate}
\end{proof}
\begin{theorem}
\label{teo:Dn_iff_Dcyclegluing}Let $G$ be any bigraph with $n\ge4$
vertices. Then $G$ has Dynkin type $\D_{n}$ if and only if it is
a $\D$-cycle gluing.\end{theorem}
\begin{svmultproof}
$\Rightarrow$) We prove this implication by induction on $k$ of
Lemma \ref{lem:flations_method}. For the base case, the $\D$-cycle
gluing of $G^{\left(0\right)}=\D_{n}$ is given by: \newcommand{\Drightside}{\begin{tikzpicture}
[baseline=(v3).base, scale = 1.0, line width=0.5pt, every node/.style={circle, inner sep=0pt}]
    \node (v2) at (0.13, -0.5) {$v_2$};
    \node (v3) at (1, 0) {$v_3$};
	\node (v4) at (2, 0) {$v_4$};
	\node (v5) at (5, 0) {$v_n$};
	\node (v1) at (0.13, 0.5) {$v_1$};
	\node[draw = none] (dots) at (3.5, 0) {$\ldots$};
	\draw (v2) -- (v3) -- (v4) -- (3, 0);
	\draw (v1) -- (v3);
	\draw (4, 0) -- (v5);
	\draw[dotted] (v2)--(v1);
\end{tikzpicture}}\newcommand{\Dcenterside}{\begin{tikzpicture}
[baseline=0.5, scale = 1.0, line width=0.5pt, every node/.style={circle, fill=white, inner sep=0pt}]
	\draw[dotted] (0,-0.5) arc [start angle=-90, end angle=90, radius=0.5];
	\draw (0,0.5) arc [start angle=90, end angle=270, radius=0.5];
    \node (v2) at (0, -0.5) {$v_2$};
	\node (v1) at (0, 0.5) {$v_1$};
\end{tikzpicture}}\newcommand{\Dleftside}{\begin{tikzpicture}
[baseline=0.5, scale = 1.0, line width=0.5pt, every node/.style={circle, inner sep=0pt}]
    \node (v2) at (0, -0.5) {$v_2$};
	\node (v1) at (0, 0.5) {$v_1$};
	\draw (v2)--(v1);
\end{tikzpicture}}

\[
F_{1}=\Dleftside\qquad H=\Dcenterside\qquad F_{2}=\Drightside
\]

For the induction step, let $G^{\left(k\right)}=T_{s\, r}\, G^{\left(k-1\right)}$
where $G^{\left(k-1\right)}$ is a $\D$-cycle gluing. Then by Lemma
\ref{lem:T_preserves_Dcycle} $G^{\left(k\right)}$ is also a $\D$-cycle
gluing.

$\Leftarrow$) Let $G$ be any $\D$-cycle gluing of $H$ and $F_{1},F_{2},\ldots,F_{h}$
and $n=\left|V\left(G\right)\right|\ge4$. If $\bigcup\left\{ F_{1},\ldots,F_{h}\right\} \ne H$
then there exists a vertex $s$ joined to some $x_{i}$ by an edge
$\left\{ s,x_{i}\right\} \in E\left(F_{i}\right)$ for some $i$.
Since $x_{i}$ is not a separation vertex of $F_{i}$, then $\left\{ s,x_{i+1}\right\} \in E\left(F_{i}\right)$.
Using case 3 of the proof of Lemma \ref{lem:T_preserves_Dcycle} repeatedly
we grow the cycle until we get to a bigraph $H^{\prime}$ which is
a $\D$-cycle gluing of $H^{\prime}$ itself and $F_{1},F_{2},\ldots,F_{n}$,
where each $F_{i}=\F\left[\left\{ x_{i}\right\} ,\left\{ x_{i+1}\right\} \right]$
if $\left\{ x_{i},x_{i+1}\right\} $ is $\solid$ in $H^{\prime}$,
and $F_{i}=\F\left[\left\{ x_{i},x_{i+1}\right\} ,\emptyset\right]$
otherwise. Notice that the transformation $I_{x_{i}}=T_{x_{i}\, x_{i+1}}T_{x_{i}\, x_{i-1}}$
has the effect of changing $\left\{ x_{i-1},x_{i}\right\} $ and $\left\{ x_{i},x_{i+1}\right\} $
to their opposite $\solid$/$\dotted$ line style. Since $H^{\prime}$
has an odd number of $\dotted$ edges, then applying $I_{x_{i}}$
whenever $\left\{ x_{i-1},x_{i}\right\} $ is $\dotted$ for $i=2,3,\ldots,n$
yields a cycle bigraph $H^{\prime\prime}$ where $\left\{ x_{n},x_{1}\right\} $
is the only $\dotted$ edge. Finally, let
\[
H^{\prime\prime\prime}=T_{x_{3}\, x_{1}}T_{x_{4}\, x_{1}}\,\ldots\, T_{x_{n}\, x_{1}}\, H^{\prime\prime}
\]
Then $H^{\prime\prime\prime}=\D_{n}$ with $x_{i}=v_{i}$.
\end{svmultproof}
We can easily prove a construction that resembles very much that of
\cite{Barot2001}. We start with the following lemma:
\begin{lemma}
\label{lem:An_shortest_paths}Let $G$ be a bigraph of Dynkin type
$\A_{n}$, then the shortest path between any pair of vertices $x$
and $y$ is unique. Moreover, it is $x\, s_{1}\, s_{2}\,\ldots\, s_{k}\, y$
where each $s_{i}$ is a separation vertex of $G$ in any simple path
$x\leadsto y$.\end{lemma}
\begin{svmultproof}
Let $V$ and $\mathcal{B}$ be the vertex set and block set of $G$
respectively and define the bipartite graph $T\left[V,\mathcal{B}\right]$
where each vertex $u\in V$ is joined to $B\in\mathcal{B}$ whenever
$u\in V\left(B\right)$. Clearly $T$ is a tree and contains $\BT\left(G\right)$.
Therefore the path $x\leadsto y$ is unique in $T$ and has the form
$x\, B_{0}\, s_{1}\, B_{1}\, s_{2}\,\cdots\, s_{k}\, B_{k}\, y$ where
each $s_{i}$ is a separation vertex of $G$. Thus, any simple $x\leadsto y$
path must contain vertices $\left\{ x,s_{1},\ldots,s_{k},y\right\} $.
Also, by Theorem \ref{teo:An_iff_Ablocktrees}, each $B_{i}$ contains
edges $\left\{ s_{i-1},s_{i}\right\} $ where $s_{0}=x$ and $s_{k}=y$.
Therefore $x\, s_{1}\,\cdots\, s_{k-1}\, y$ is actually a simple
path of $G$ which vertices are contained in $\left\{ x,s_{1},\ldots,s_{k-1},y\right\} $.
\end{svmultproof}
For any bigraph $G$, let $\dist_{G}\left(u,v\right)$ denote the
\emph{distance} of $u$ and $v$, given by the number of edges of
the shortest path $u\leadsto v$ in $G$. The \emph{identification}
of two nonadjacent vertices $x$ and $y$ of a bigraph $G$ is the
operation of replacing these vertices by a single vertex incident
to all the edges which are incident to either $x$ or $y$. Denote
the result of identifying $x$ and $y$ in $G$ as $\nicefrac{G}{\left\{ x,y\right\} }$.
\begin{theorem}
A bigraph $G$ with $n\ge4$ vertices has Dynkin type $\D_{n}$ if
and only if there exists a bigraph $J$ of Dynkin type $\A_{n+1}$
with two non-separating vertices $u$ and $v$ such that $\dist_{J}\left(u,v\right)\ge2$,
the shortest path $u\leadsto v$ has an odd number of $\dotted$ edges,
and $G=\nicefrac{J}{\left\{ u,v\right\} }$.\end{theorem}
\begin{svmultproof}
$\Rightarrow$) By Theorem \ref{teo:Dn_iff_Dcyclegluing} say that
$G$ is a $\D$-cycle-gluing of $H$ and $F_{1},\ldots,F_{h}$. Taking
the union of $F_{1},F_{2},\ldots,F_{h-1},F_{h}^{\prime}$ where $F_{h}^{\prime}$
is the result of replacing $x_{1}$ by a new vertex $x_{h+1}$ in
$F_{h}$ yields a bigraph $J$ with $n+1$ vertices. From the definition
of a $\D$-cycle gluing follows that each $x_{i}$ is a separation
vertex of $J$ for $i\in\left\{ 2,\ldots,h\right\} $ and that $u=x_{1}$
and $v=x_{h+1}$ are non-separating vertices of $J$. By Theorem \ref{teo:An_iff_Ablocktrees}
we infer that $J$ has Dynkin type $\A_{n+1}$. By Lemma \ref{lem:An_shortest_paths},
the shortest path $x_{1}\leadsto x_{h+1}$ in $J$ is precisely $x_{1}\, x_{2}\,\ldots\, x_{h}\, x_{h+1}$,
which by definition of $H$ has length $d_{J}\left(x_{1},x_{h+1}\right)\ge2$
and an odd number of $\dotted$ edges. Finally by the very definition
of $x_{h+1}$, identifying $x_{1}=x_{h+1}$ yields $G$.

$\Leftarrow$) Say that the shortest path $P$ from $u$ to $v$ is
$u=x_{1}\, x_{2}\,\ldots\, x_{h}\, x_{h+1}=v$. By Lemma \ref{lem:An_shortest_paths}
we see that $x_{i}$ is a separation vertex for $i\in\left\{ 2,\ldots,h\right\} $.
Let $J_{1}=J$. By Theorem \ref{teo:An_iff_Ablocktrees} $x_{i+1}$
separates $J_{i}$ into two subgraphs $F_{i}$ and $J_{i+1}$ where
$\left\{ x_{i},x_{i+1}\right\} \in E\left(F_{i}\right)\setminus E\left(J_{i+1}\right)$
for $i=1,2,\ldots,h-1$. Let $F_{h}=J_{h}$, $x_{h+1}=x_{1}$ and
$H=\nicefrac{P}{\left\{ x_{i+1},x_{1}\right\} }$. Then by construction
$G=\nicefrac{J}{\left\{ u,v\right\} }$ is a $\D$-cycle gluing of
$H$ and $F_{1},F_{2},\ldots,F_{h}$. Therefore by Theorem \ref{teo:Dn_iff_Dcyclegluing},
$G$ has Dynkin type $\D_{n}$.
\end{svmultproof}
\appendix
\bibliographystyle{spmpsci}
\phantomsection\addcontentsline{toc}{section}{\refname}\bibliography{qcartan_sym}

\end{document}